\documentclass[11pt]{amsart}
\usepackage{mathrsfs}
\usepackage{amsfonts}
\usepackage{amsthm}
\usepackage{amssymb}
\usepackage{latexsym}
\usepackage{color}
\usepackage{xcolor}
\usepackage{soul}
\newtheorem{theorem}{Theorem}
\newtheorem{lemma}{Lemma}
\newtheorem{corollary}{Corollary}

\newtheorem{proposition}{Proposition}

\begin{document}

\title[Diederich--Forn\ae ss index and global regularity]{Diederich--Forn\ae ss index and global regularity in the $\overline{\partial}$--Neumann problem: domains with comparable Levi eigenvalues}

\author{Bingyuan Liu}
\author{Emil J.~Straube}

\address{School of Mathematical and Statistical Sciences, University of Texas Rio Grande Valley Edinburg, Texas, 78539}
\email{bingyuan.liu@utrgv.edu}

\address{Department of Mathematics, Texas A\&M University College Station, Texas, 77843}
\email{e-straube@tamu.edu}


\date{July 29, 2022, revision March 19, 2025}

\begin{abstract}
Let $\Omega$ be a smooth bounded pseudoconvex domain in $\mathbb{C}^{n}$. Let $1\leq q_{0}\leq (n-1)$. We show that if $q_{0}$--sums of eigenvalues of the Levi form are comparable, then if the Diederich--Forn\ae ss index of $\Omega$ is $1$, the $\overline{\partial}$--Neumann operators $N_{q}$ and the Bergman projections $P_{q-1}$ are regular in Sobolev norms for $q_{0}\leq q\leq n$. In particular, for domains in $\mathbb{C}^{2}$, Diederich--Forn\ae ss index $1$ implies global regularity in the $\overline{\partial}$--Neumann problem.
\end{abstract}

\maketitle

\section{Introduction}\label{intro}
Let $\Omega$ be a bounded pseudoconvex domain in $\mathbb{C}^{n}$ with ($C^{\infty}$) smooth boundary. Then there exists $\eta\in (0,1)$ and a defining function $\rho_{\eta}$ such that $-(-\rho_{\eta})^{\eta}$ is plurisubharmonic on $\Omega$ near $b\Omega$ (\cite{DF77b},\cite{Range81}). The supremum over all such $\eta$ has come to be known as the Diederich--Forn\ae ss index of $\Omega$. The relevance of the index stems from the fact that in general, $\Omega$ does not admit a defining function that is plurisubharmonic on $\Omega$ near the boundary, not even locally (\cite{Fornaess79}, \cite{Behrens85}, \cite{GallagherHarz22}).

The index is known to be strictly less than one on the Diederich--Forn\ae ss worm domains (\cite{DF77b}; see \cite{Liu19} for the exact value; see also \cite{AbdulHarr19}). It is known to be equal to one in the following cases: $\Omega$ admits a defining function that is plurisubharmonic \emph{at} (not necessarily near) the boundary (\cite{FornaessHerbig08}), there are `good vector fields' and the set of infinite type points is `well behaved' (\cite{Harrington19}), $b\Omega$ satisfies Property(P) (\cite{Harrington19}), $\Omega$ is strictly pseudoconvex except for a simply connected complex manifold in the boundary (\cite{Liu19b}). There are two things to note about this list. First, index one does not imply that there is a defining function that is plurisubharmonic on $\Omega$ (near $b\Omega$). Indeed, domains with real analytic boundaries are of finite type, so satisfy property(P), yet need not admit even local plurisubharmonic defining functions (\cite{Behrens85, GallagherHarz22}). Second, and perhaps more strikingly, all domains in the list with index one are known to have globally regular Bergman projections and $\overline{\partial}$--Neumann operators (\cite{BoasStraube91a, BoasStraube93, Catlin82, Straube10a, Harrington19}), while on the worm domains, these operators are regular only up to a finite Sobolev level that is closely related to their index (see \cite{BoasStraube89, Barrett92, Christ96, Liu19}). 

\smallskip

The question what the implications of the Diederich--Forn\ae ss index are for global regularity arose in earnest after the results on the worm domains in \cite{BoasStraube89, Barrett92, Christ96} and after \cite{BoasStraube91a}, where the authors showed that if $\Omega$ admits a defining function that is plurisubharmonic at the boundary, then the Bergman projections and $\overline{\partial}$--Neumann operators (at all form levels) on $\Omega$ are globally exactly regular, i.e., are continuous in Sobolev--$s$ norms for $s\geq 0$. A quantitative study of this question was initiated in \cite{Kohn99} where it was shown that if the index is one, and there is some control on the defining functions $\rho_{\eta}$ as $\eta\rightarrow 1^{-}$, then global regularity holds. More precisely, the condition is that $\liminf_{(\eta\rightarrow 1^{-})}((1-\eta)^{1/3}\max_{b\Omega}|\nabla h_{\eta}|) = 0$, where $h_{\eta}$ is given by $\rho_{\eta}=e^{h_{\eta}}\rho$ for some fixed defining function $\rho$. The exponent $1/3$ was improved to $1/2$ in \cite{Harrington11}. In \cite{BernChar00}, it is shown that if the index is $\eta_{0}\leq 1$, then regularity holds for $0\leq s<\eta_{0}/2\;(\leq 1/2)$. In \cite{PinZamp14}, the authors, among other things, generalize the ideas from \cite{Kohn99} to $q$--convex domains, with a corresponding notion of the index. More recently, the first author showed in \cite{Liu22} that index one implies global regularity for domains with comparable eigenvalues of the Levi form without assuming any control on the defining functions $\rho_{\eta}$, but instead making a technical assumption that controls the geometry of the set of infinite type boundary points.

\smallskip

In this paper, we show, via a different proof, that this assumption is not needed. We also consider the assumption at the level of $q$--forms for $q>1$. In this case, it says that sums of $q$ eigenvalues of the Levi form should be comparable; see section \ref{pre} for a precise definition and discussion.
\begin{theorem}\label{main}
 Let $\Omega$ be a smooth bounded pseudoconvex domain in $\mathbb{C}^{n}$, $1\leq q_{0}\leq (n-1)$. Assume that $q_{0}$--sums of the eigenvalues of the Levi form are comparable. Then, if the Diederich-Forn\ae ss index of $\Omega$ is $1$, the Bergman projections $P_{q-1}$ and the $\overline{\partial}$--Neumann operators $N_{q}$, $q_{0}\leq q\leq n$, are continuous in Sobolev--s norms for $s\geq 0$. 
 \end{theorem}
When $q=(n-1)$, there is only one $q$--sum, and the comparability condition is trivially satisfied. 
 \begin{corollary}\label{n-1}
  Let $\Omega$ be a smooth bounded pseudoconvex domain in $\mathbb{C}^{n}$. If the Diederich-Forn\ae ss index of $\Omega$ is $1$, the Bergman projection $P_{n-2}$ and the $\overline{\partial}$--Neumann operators $N_{n-1}$ are continuous in Sobolev--s norms for $s\geq 0$.
 \end{corollary}
 For emphasis, we single out the case $n=2$, obtaining regularity for $P_{0}$ and $N_{1}$, the cases usually considered the most important.
 \begin{corollary}\label{C2}
  Let $\Omega$ be a smooth bounded pseudoconvex domain in $\mathbb{C}^{2}$. If the Diederich-Forn\ae ss index of $\Omega$ is $1$, the Bergman projection $P_{0}$ and the $\overline{\partial}$--Neumann operator $N_{1}$ are continuous in Sobolev--s norms for $s\geq 0$. 
 \end{corollary}
 \emph{Remark:} Two smooth bounded pseudoconvex domains in $\mathbb{C}^{n}$ which are biholomorphic via a biholomorphism that extends to a diffeomorphism of the closures have the same Diederich--Forn\ae ss index. An immediate consequence of Corollary \ref{C2} is therefore that if two smooth bounded pseudoconvex domains in $\mathbb{C}^{2}$ are biholomorphic, then if one has index $1$, so does the other (the biholomorphism extends, by \cite{Bell81}).

 \smallskip

The main ingredients in our proof of Theorem \ref{main} are first recent work in \cite{Liu19}, as reformulated in \cite{Yum21}, where it is shown how to express the Diederich--Forn\ae ss index in terms of estimates on D'Angelo forms $\alpha_{\eta}:=\alpha^{\rho_{\eta}}$ associated with a defining function $\rho_{\eta}$ of the domain (see section \ref{pre}). The resulting estimates on $\alpha_{\eta}$ then involve $\sum_{j,k}(\partial^{2}h_{\eta}/\partial z_{j}\partial\overline{z_{k}})u_{j}\overline{u_{k}}$, where $h_{\eta}$ is as above (see \eqref{hessian} below). The second important point is the observation that one does not need pointwise estimates on $|\alpha_{\eta}|$, much weaker $L^{2}$--type estimates suffice (compare \cite{Straube05}). So instead of controlling the Hessian of $h_{\eta}$ pointwise, which seems hopeless, one only has to deal with  $\int_{\Omega}\sum_{j,k}(\partial^{2}h_{\eta}/\partial z_{j}\partial\overline{z_{k}})u_{j}\overline{u_{k}}$. This latter expression is familiar in the $L^{2}$ theory of the $\overline{\partial}$--Neumann problem, and one can tweak known machinery to obtain a strong estimate on $\int_{\Omega}|\alpha_{\eta}(L_{u})|^{2}$ (Proposition \ref{integrated} below).  From the point of view of exploiting the Diederich--Forn\ae ss index, this is the central estimate. Once this estimate is in hand, the proof of Theorem \ref{main} initially follows \cite{Straube05}, but then uses the formulas for the commutators of $\overline{\partial}$ and $\overline{\partial}^{*}$ with the `usual' transversal vector fields found in \cite{HarringtonLiu20} (Lemmas \ref{dbar} and \ref{dbar*} below).

\smallskip

The rest of the paper is organized as follows. Section \ref{pre} contains definitions and notation, as well as some preliminary lemmas. Section \ref{index-form} contains Proposition \ref{integrated} and its proof. The proof of Theorem \ref{main} is given in section \ref{proof}.

 \section{Notation and Preliminaries}\label{pre}
 
 Let $\Omega$ be a bounded domain in $\mathbb{C}^{n}$ with ($C^{\infty}$) smooth boundary. We denote by $P_{q}$ the Bergman projection on $(0,q)$--forms, $0\leq q\leq n$, and by $N_{q}$ the $\overline{\partial}$--Neumann operator on $(0,q)$--forms, $1\leq q\leq n$. We refer the reader to \cite{ChenShaw01} and \cite{Straube10a} for background on these topics, as well as for standard notation.
 
 \smallskip
 
 We use the following notation for weighted Sobolev norms: for a real valued function $g$ so that $e^{-g} \in C^{\infty}(\overline{\Omega})$, we set $\|u\|_{k,g} = \left(\sum_{s=0}^{k}\int_{\Omega}|\nabla^{s}u|^{2}e^{-g}\right)^{1/2}$, where $\nabla^{s}$ denotes the vector of all derivatives of order $s$. If $u$ is a form, the derivatives act on coefficients as usual (in Euclidean coordinates, unless stated otherwise).
 
 \smallskip
 
  It will be important that certain constants do not depend on $\eta$. We adopt the convention to denote such constants with $C$ and we allow the actual value to change from one occurrence to the next. On the other hand, constants that do depend on $\eta$ will be denoted by $C_{\eta}$, with the same convention about the actual value. We will similarly use subscripts to denote dependence on the form level, or on both form level and $\eta$. When it becomes cumbersome to write $C(\cdots)$, we will use the symbol $\lesssim$ instead.
 
\smallskip
 
 It is by now a standard fact that derivatives of type $(0,1)$ and complex tangential derivatives of either type are benign for the $\overline{\partial}$--Neumann problem and the Bergman projection (\cite{BoasStraube91a}, \cite{Straube10a}, Lemma 5.6). The proof of Theorem \ref{main} requires a version for weighted norms, which we formulate here for the reader's convenience; it follows readily from the unweighted version.

\begin{lemma}\label{benign}
 Let $\Omega$ be a smooth bounded pseudoconvex domain in $\mathbb{C}^{n}$, $k\in \mathbb{N}$, and $Y$ a vector field of type $(1,0)$ with coefficients in $C^{\infty}(\overline{\Omega})$ that is tangential on the boundary. There exists a constant $C$ such that when the weight $e^{-h/2}\in C^{\infty}(\overline{\Omega})$, there exists a constant $C_{h}$ so that for $u\in C^{\infty}_{(0,q)}(\overline{\Omega})\cap dom(\overline{\partial}^{*})$ we have the estimates
 \begin{equation}\label{benign1}
  \sum_{j,J}\left\|\frac{\partial u_{J}}{\partial\overline{z_{j}}}\right\|_{k-1,h}^{2} \leq C\left(\|\overline{\partial}u\|_{k-1,h}^{2}+\|\overline{\partial}^{*}u\|_{k-1,h}^{2}\right) + C_{h}\|u\|_{k-1,h}^{2}\;,
 \end{equation}
and
\begin{equation}\label{benign2}
 \|Yu\|_{k-1,h}^{2} \leq C\left(\|\overline{\partial}u\|_{k-1,h}^{2}+\|\overline{\partial}^{*}u\|_{k-1,h}^{2}\right) + C_{h}\|u\|_{k-1,h}\|u\|_{k,h} \;.
\end{equation}
\end{lemma}
Writing weights as exponentials is convenient in the context of Kohn--Morrey--H\"{o}rmander type formulas, but somewhat artificial in the lemma, as only the smoothness of $e^{-h/2}$ matters. We keep the convention so we do not need to introduce additional notation. 
\begin{proof}[Proof of Lemma \ref{benign}]
 Apply the unweighted version (\cite{Straube10a}, Lemma 5.6) to the form $e^{-h/2}u\in dom(\overline{\partial}^{*})$ and observe that if a derivative $D$ hits $e^{-h/2}$, the resulting term is of the form $D(e^{-h/2})u$ and is thus dominated by $C_{h}\|u\|_{k-1,h}^{2}$ and $C_{h}\|u\|_{k-1,h}\|u\|_{k,h}$, respectively.
 \end{proof}

\smallskip

There are two points in the proof of Theorem \ref{main} where we need  an (unweighted) estimate that is stronger than \eqref{benign2}, namely (with $Y$ as in Lemma \ref{benign})
\begin{equation}\label{max}
 \|Yu\|_{k-1}^{2} \leq C\left(\|\overline{\partial}u\|_{k-1}^{2}+\|\overline{\partial}^{*}u\|_{k-1}^{2}\right) + C\|u\|_{k-1}^{2} \;.
\end{equation}
Such estimates, referred to as maximal estimates, play an important role in the theory of the $\overline{\partial}$--Neumann problem (see for example the introduction in \cite{Koenig15}).
We will however use Lemma \ref{benign} where possible, so as to minimize the use of these stronger estimates.

The necessary and sufficient condition for maximal estimates to hold is that $q$--sums of eigenvalues of the Levi eigenvalues dominate the trace of the Levi form (\cite{Derridj78}, Th\'{e}or\`{e}me 3.1 for $q=1$, \cite{BenMoussa00}, Th\'{e}or\`{e}me 3.7 for $q>1$). More precisely, this condition means that there is a constant $C$ such that for any point $p\in b\Omega$ and $j_{1},\cdots,j_{q}$ with $1\leq j_{1}<\cdots<j_{q}\leq (n-1)$, it holds that $C(\lambda_{1}(p)+\lambda_{2}(p)+\cdots+\lambda_{n-1}(p)) \leq \lambda_{j_{1}}(p)+\cdots+\lambda_{j_{q}}(p)$, where $\lambda_{1}(p)$, $\cdots$, $\lambda_{(n-1)}(p)$ denote the eigenvalues of the Levi form at $p$, listed with multiplicity These eigenvalues are independent of the basis as long as they are taken with respect to an orthonormal basis\footnote{In \cite{CSS20}, footnote 1, the authors point out that thanks to a theorem of Ostrowski (\cite{HornJohnson85}, Theorem 4.5.9) the comparability condition remains independent of the basis even when the orthonormality  requirement is dropped.}. Because $\Omega$ is pseudoconvex, all Levi eigenvalues are non negative, and the previous condition is easily seen to be equivalent to the following one: there exists a constant $C$ such that for  any pair of $q$--tuples $(j_{1}, \cdots, j_{q})$ and $(k_{1}, \cdots, k_{q})$, it holds that $\lambda_{j_{1}}(p)+\cdots+\lambda_{j_{q}}(p) \leq C(\lambda_{k_{1}}(p)+ \cdots+\lambda_{k_{q}}(p))$ (by an argument similar to the one in the proof of Lemma \ref{percup} below). We say for short that $q$--sums of the Levi eigenvalues of $\Omega$ are comparable. Note that for $q=(n-1)$, there is only one $q$--sum, and the comparability condition trivially holds. In particular, the condition trivially holds for domains in $\mathbb{C}^{2}$ and $q=1$.

We will need that the comparability condition on the Levi eigenvalues percolates up the Cauchy--Riemann complex:
\begin{lemma}\label{percup}
 $\Omega$ as above. Assume the comparability condition for $q$--sums of eigenvalues of the Levi form is satisfied for some level $q$. Then it is satisfied at level $(q+1)$.
\end{lemma}
\begin{proof}
 The (standard) argument is to observe that for $\{j_{1}, \cdots, j_{q+1}\}$ fixed, $\lambda_{j_{1}}(p)+\cdots\lambda_{j_{q+1}}(p)$ $= (1/q)\sum_{\{k_{1},\cdots, k_{q}\}\subset\{1,\cdots, q+1\}}(\lambda_{j_{k_{1}}}(p)+\cdots+\lambda_{j_{k_{q}}})$.
\end{proof}

\smallskip

Fix a defining function $\rho$ for $\Omega$. Near $b\Omega$, set $L_{n}^{\rho}=(1/\sum_{j}|\partial\rho/\partial z_{j}|^{2})\sum_{j=1}^{n}(\partial\rho/\partial\overline{z_{j}})\partial/\partial z_{j}$, and continue it smoothly to all of $\Omega$. Then $L_{n}^{\rho}\rho \equiv 1$ near $b\Omega$ (so $L_{n}^{\rho}$ is in general normalized differently from the unit normal $L_{n}$). We also define $T^{\rho}:=L_{n}^{\rho}-\overline{L_{n}^{\rho}}$, and $\sigma^{\rho}:=(1/2)(\partial\rho-\overline{\partial}\rho)$; note that $\sigma^{\rho}(T^{\rho})\equiv 1$ near $b\Omega$. The D'Angelo $1$--form is defined as $\alpha^{\rho}= -\mathcal{L}_{T^{\rho}}\sigma^{\rho}$, where $\mathcal{L}_{T^{\rho}}$ denotes the Lie derivative in the direction of $T^{\rho}$. Although $\alpha^{\rho}$ depends on the defining function $\rho$, all its relevant properties are intrinsic (that is, do not depend on the choice of $\rho$). We will use various of these properties, all of which can be found in \cite{Straube10a}. In particular, if $L=\sum_{k=1}^{n}w_{k}(\partial/\partial z_{k})$ is a local section of $T^{(1,0)}(b\Omega_{\varepsilon})$, where $\Omega_{\varepsilon} = \{z\in\Omega\;|\;\rho(z)<-\varepsilon\}$, $0\leq\varepsilon$ small, then
 \begin{equation}\label{alpha0}
  \alpha^{\rho}(\overline{L}) = \partial\rho([L_{n}^{\rho},\overline{L}]) = \frac{1}{\sum_{j}|\partial\rho/\partial z_{j}|^{2}}\sum_{j,k=1}^{n}\frac{\partial^{2}\rho}{\partial z_{j}\partial\overline{z_{k}}}\,\frac{\partial\rho}{\partial\overline{z_{j}}}\,\overline{w_{k}}\;;
 \end{equation}
(see \cite{Straube10a}, (5.76) and (5.85)). Moreover, if $\rho_{h}=e^{h}\rho$ is another defining function, then
\begin{equation}\label{alpha00}
 \alpha_{h}(L):=\alpha^{\rho_{h}}(L) = \alpha^{\rho}(L) + \partial h(L)
\end{equation}
(\cite{Straube10a}, (5.84)). Also note that $\alpha^{\rho}(T^{\rho})=0$ (from the definition and the normalizations).

The forms $\alpha^{\rho}$ play a crucial role for Sobolev estimates in the $\overline{\partial}$--Neumann problem in that they control commutators of certain vector fields with $\overline{\partial}$ and $\overline{\partial}^{*}$(compare \eqref{alpha0} above, \cite{Straube10a}, \cite{HarringtonLiu20},\cite{DAM}). We recall two lemmas from \cite{HarringtonLiu20} that give explicit formulas for these commutators, modulo benign errors. These formulas are stated for forms supported in a special boundary chart, with the usual orthonormal boundary frame $L_{1},\cdots, L_{n}$ and dual frame $\omega_{1}, \cdots, \omega_{n}$.
Vector fields (differential operators) on forms are acting  coefficient wise in this special chart. Let $v=\sideset{}{'}\sum_{J}v_{J}\overline{\omega_{J}}\in C^{\infty}_{(0,q)}(\overline{\Omega})$, and let $0\leq q\leq n$. In our notation, and with our normalizations, the formulas from \cite{HarringtonLiu20} are as follows.
\begin{lemma}[\cite{HarringtonLiu20}, Lemma 3.1]\label{dbar}
 Let $k\in\mathbb{N}$, $h\in C^{\infty}(\overline{\Omega})$, real valued, and $\rho$ a defining function for $\Omega$. Then
 \begin{multline}\label{dbarest}
 [\overline{\partial},(e^{-h}(L_{n}^{\rho}-\overline{L_{n}^{\rho}}))^{k}]v \\
 = -k\sum_{j}\sideset{}{'}\sum_{J}(\overline{L_{j}}h+\alpha^{\rho}(\overline{L_{j}}))((e^{-h}(L_{n}^{\rho}-\overline{L_{n}^{\rho}}))^{k}v_{J})\;\overline{\omega_{j}}\wedge\overline{\omega_{J}} + A_{h}(v) + B_{h}(v)  \;,
 \end{multline}
where $A_{h}(v)$ consists of terms of order $k$ with at least one of the derivatives being barred or complex tangential, and $B_{h}(v)$ consist of terms of order at most $(k-1)$.
\end{lemma}
We also have
\begin{lemma}[\cite{HarringtonLiu20}, Lemma 3.2]\label{dbar*}
With the assumptions as in Lemma \ref{dbar}, but additionally $v\in dom(\overline{\partial}^{*})$ and $q\geq 1$. Then
\begin{multline}\label{dbar*est}
 [\overline{\partial}^{*},(e^{-h}(L_{n}^{\rho}-\overline{L_{n}^{\rho}}))^{k}]v \\
 = -k\sum_{j}\sideset{}{'}\sum_{S}(L_{j}h+\alpha^{\rho}(L_{j}))((e^{-h}(L_{n}^{\rho}-\overline{L_{n}^{\rho}}))^{k}v_{jS})\;\overline{\omega_{S}} + A_{h}(v) + B_{h}(v) ,
\end{multline}
where $A_{h}(v)$ and $B_{h}(v)$ are of the same type as in Lemma \ref{dbar} (but different operators). 
\end{lemma}
Note that when $j<n$, then $\overline{L_{j}}h+\alpha^{\rho}(\overline{L_{j}})$ equals $\alpha_{h}(\overline{L_{j}})$ (in view of \eqref{alpha00}). Likewise, $L_{j}h+\alpha^{\rho}(L_{j}) = \alpha_{h}(L_{j})$ when $j<n$. Thus the form $\alpha_{h}$ controls the commutators in \eqref{dbarest} and \eqref{dbar*est}, except for the term with $j=n$ in \eqref{dbarest}. However, it has been understood at least since \cite{BoasStraube91a} that control of this term comes for free. The error terms $A_{h}(v)$, $B_{h}(v)$, and $C_{h}(v_{N})$ are normally under control ($B_{h}(v)$ is lower order, for $A_{h
}(v)$ use Lemma \ref{benign}, and the normal component acts as a subelliptic multiplier, see estimate (2.96) in \cite{Straube10a}; Proposition 4.7, part (G), in \cite{Kohn79}). In fact, in \cite{HarringtonLiu20} the error terms are given only in this estimated form. On the other hand, for the proof of the lemmas, which results from computing the case $k=1$ in the boundary frame, and then using induction on $k$, it is convenient to have an expression for the terms. The factor $k$ on the right hand sides of \eqref{dbarest} and \eqref{dbar*est} arises in the induction from the commutator formula $[A, T^{k}] = \sum_{j=1}^{k}\binom{k}{j}[\cdots[[A,T],T]\cdots]\,T^{k-j} = k\,[[A, T], T]T^{k-1}$ $+$ terms of order at most $(k-1)$ (see \cite{DerridjTartakoff76}, Lemma 2 on page 418, or \cite{Straube10a}, formula (3.54)).

\smallskip

Denote by $DF(\Omega)$ the Diederich--Forn\ae ss index of the domain $\Omega$. Then, if $0 <\eta < DF(\Omega)$, there is a defining function $\rho_{\eta}$ so that $-(-\rho_{\eta})^{\eta}$ is plurisubharmonic in $\Omega$ near the boundary. Namely, by the definition of the index, there is $\eta_{1}$, $\eta < \eta_{1} < DF(\Omega)$, and a defining function $\rho_{\eta_{1}}$, such that $-(-\rho_{\eta_{1}})^{\eta_{1}}$ is plurisubharmonic near the boundary. Then so is $-(-\rho_{\eta_{1}})^{\eta}$ (compose $-(-\rho_{\eta_{1}})^{\eta_{1}}$ with the increasing convex function $-(-x)^{\eta/\eta_{1}}$, $x \leq 0$); thus $\rho_{\eta}:=\rho_{\eta_{1}}$ will do for $\eta$. For such $\eta$, there is then a function $h_{\eta}\in C^{\infty}(\overline{\Omega})$ such that near the boundary $\rho_{\eta} = e^{h_{\eta}}\rho$ ($\rho$ is the fixed defining function from above).

\section{Diederich-Forn\ae ss index and estimates on D'Angelo forms}\label{index-form}

Our proof of Theorem \ref{main} depends on recent work in \cite{Liu19}, as reformulated in \cite{Yum21}, where it is shown that the Diederich--Forn\ae ss index of a domain can be expressed in terms of the form $\alpha_{\eta}:=\alpha^{\rho_{\eta}}$. Denote by $\omega_{\eta}$ the $(1,0)$--part of $\alpha_{\eta}$; then $\alpha_{\eta} = \omega_{\eta}+\overline{\omega_{\eta}}$ (there is no ambiguity as to the these parts for the real one form $\alpha_{\eta}$ on the boundary, because $\alpha_{\eta}(T)=0$, see also the discussion in Section 4.3 of \cite{DAM}). The formulation most convenient for us is in Theorem 1.1 in \cite{Yum21}:
\begin{equation}\label{DFalpha}
 DF(\Omega) = \sup\left\{0<\eta<1:\left(\frac{\eta}{1-\eta}(\omega_{\eta}\wedge\overline{\omega_{\eta}})-\overline{\partial}\omega_{\eta}\right)(L\wedge\overline{L}) \leq 0\;;\;p\in \Sigma\;,\; L\in\mathcal{N}_{p}\right\}\;,
\end{equation}
where $\Sigma\subset b\Omega$ denotes the set of weakly pseudoconvex boundary points, $\mathcal{N}_{p}$ denotes the null space of the Levi form at $p\in b\Omega$, and the supremum is over all $\eta$ so that there exists a defining function $\rho_{\eta}$ such that the inequality holds. Because $|\alpha_{\eta}(L)|^{2}=(\omega_{\eta}\wedge\overline{\omega_{\eta}})(L\wedge\overline{L})$ and $\overline{\partial}\omega_{\eta}(L\wedge\overline{L})=\overline{\partial}\alpha_{\eta}(L\wedge\overline{L})$, \eqref{DFalpha} implies $|\alpha_{\eta}(L(p))|^{2} \lesssim (1-\eta)\overline{\partial}\alpha_{\eta}(L(p)\wedge\overline{L(p)})\;;\;p\in \Sigma\;,\;L(p)\in\mathcal{N}_{p}$. Section 4.3 in \cite{DAM}  (see in particular Lemmas 4.6 and 4.7) gives that for $p\in \Sigma$, $L(p)\in\mathcal{N}_{p}$, we have $\overline{\partial}\alpha_{\eta}(L(p)\wedge\overline {L(p)})=\overline{\partial}\alpha^{\rho}(L(p)\wedge\overline{L(p)})-\partial\overline{\partial}h_{\eta}(L(p)\wedge\overline{L_{p}})$. Consequently, we have, for some constant $C$ that does not depend on $\eta$,
\begin{equation}\label{hessian}
 |\alpha_{\eta}(L(p))|^{2} \leq C(1-\eta)\left(\sum_{j,k=1}^{n}\frac{\partial^{2}(-h_{\eta})}{\partial z_{j}\partial\overline{z_{k}}}L_{j}\overline{L_{k}} + |L|^{2}\right)\;;\;p\in \Sigma\;,\;L(p)\in\mathcal{N}_{p}\;;
\end{equation}
the $|L|^{2}$ term comes from $\overline{\partial}\alpha$.

To obtain an estimate when $L(p)$ is merely complex tangential at the boundary, we use that fact that in strictly pseudoconvex directions, the Levi form dominates vectors. We first make \eqref{hessian} strict for $|L|=1$ by adding $C(1-\eta)|L|^{2}$ to the right hand side:
\begin{equation}\label{hessian2}
 |\alpha_{\eta}(L(p))|^{2} < C(1-\eta)\left(\sum_{j,k=1}^{n}\frac{\partial^{2}(-h_{\eta})}{\partial z_{j}\partial\overline{z_{k}}}L_{j}\overline{L_{k}} + 2|L|^{2}\right)\,;\,p\in \Sigma,L(p)\in\mathcal{N}_{p}, |L(p)|=1\,.
\end{equation}
Denote by $S^{1,0}(b\Omega)$ the unit sphere bundle in $T^{1,0}(b\Omega)$. The set $A:=\{(p,L)\;|\;p\in\Sigma,L\in \mathcal{N}_{p}\}$ is a compact subset of $S^{1,0}(b\Omega)$, and \eqref{hessian2} holds for $(p,L)\in A$. By continuity, it holds for $(p,L)$ in an open neighborhood $U_{\eta}$ of $A$ (in $S^{1,0}(b\Omega)$). On $S^{1,0}(b\Omega)\setminus U_{\eta}$, $\sum_{j,k}\frac{\partial^{2}\rho}{\partial z_{j}\partial\overline{z_{k}}}L_{j}\overline{L_{k}}$ admits a strictly positive lower bound ($S^{1,0}(b\Omega)\setminus U_{\eta}$ is compact). Therefore, for a big enough constant $M_{\eta}$, we have
\begin{multline}\label{esta0}
 |\alpha_{\eta}(L)|^{2} < C(1-\eta)\left(\sum_{j,k}\frac{\partial^{2}(-h_{\eta})}{\partial z_{j}\partial\overline{z_{k}}}L_{j}\overline{L_{k}} + 2|L|^{2}\right) 
 + M_{\eta}\sum_{j,k}\frac{\partial^{2}\rho}{\partial z_{j}\partial\overline{z_{k}}}L_{j}\overline{L_{k}}\;,\\(p,L)\in S^{1,0}(b\Omega)\setminus U_{\eta}\;.
\end{multline}
At points of $U_{\eta}$, \eqref{hessian2} holds. Adding the nonnegative term $ M_{\eta}\sum_{j,k}\frac{\partial^{2}\rho}{\partial z_{j}\partial\overline{z_{k}}}L_{j}\overline{L_{k}}$  ($\Omega$ is pseudoconvex) keeps the inequality valid: \eqref{esta0} holds on $U_{\eta}$ as well. So \eqref{esta0} holds on $U_{\eta}$ and on $S^{1,0}(b\Omega)\setminus U_{\eta}$, so holds on $S^{1,0}(b\Omega)$, i.e. for all $(p,L)$, $p\in b\Omega$, $L\in T^{1,0}_{p}(b\Omega)$, $|L|=1$. By continuity, \eqref{esta0} implies that there is a neighborhood $W_{\eta}$ of $b\Omega$, independent of $L$, such that \eqref{esta0} holds on $W_{\eta}$, for $L$ with $|L|=1$ satisfying $\sum_{j=1}^{n}(\partial\rho/\partial z_{j})L_{j}=0$. Homogeneity then gives on $W_{\eta}$:
\begin{multline}\label{esta}
 |\alpha_{\eta}(L)|^{2} \leq C(1-\eta)\left(\sum_{j,k}\frac{\partial^{2}(-h_{\eta})}{\partial z_{j}\partial\overline{z_{k}}}L_{j}\overline{L_{k}} + 2|L|^{2}\right) \\
 + M_{\eta}\sum_{j,k}\frac{\partial^{2}\rho}{\partial z_{j}\partial\overline{z_{k}}}L_{j}\overline{L_{k}}\;;\;\; \sum_{j=1}^{n}(\partial\rho/\partial z_{j})L_{j}=0\;.
\end{multline}
Note that upon inserting \eqref{alpha00} into the left hand side of \eqref{esta}, the estimate becomes, roughly speaking, a self bounded gradient condition for $(-h_{\eta})$ with constant $(1-\eta)$, modulo terms that are either independent of $\eta$, or will turn out to be under control.

\smallskip

In the proof of Theorem \ref{main}, we will need only an integrated version of \eqref{esta}. More precisely, if for a $(0,q)$--form $u=\sideset{}{'}\sum_{K}u_{K}\overline{dz_{K}}\in C^{\infty}_{(0,q)}(\overline{\Omega})\cap dom(\overline{\partial}^{*})$, we set $L_{u}^{J}=\sum_{j}u_{jJ}(\partial/\partial z_{j})$, $J$ an increasing $(q-1)$--tuple, then we need an estimate on $\sideset{}{'}\sum_{J}\int_{\Omega}|\alpha_{\eta}(L^{J}_{u})|^{2}$. However, integrating \eqref{esta} over $\Omega$, with $L=L_{u}^{J}$, requires some care. The reason is that while on the boundary, $L^{J}_{u}\in T^{1,0}(b\Omega)$ (because $u\in dom(\overline{\partial}^{*})$), the condition on $L$ in \eqref{esta} need not hold in general away from the boundary. This situation is remedied by splitting $L^J_{u}$ into its normal and tangential parts near $b\Omega$: $L^J_{u}=(L^J_{u})_{T}+(L^J_{u})_{N}$ . We refer the reader to \cite{Straube10a}, section 2.9 for a detailed discussion of this decomposition. What is important for us is that $\|(L^J_{u})_{N}\|_{1}\leq C(\|\overline{\partial}u\|+\|\overline{\partial}^{*}u\|)$, see \cite{Straube10a}, Lemma 2.12, and that $(L^J_{u})_{T}$ satisfies the requirement on $L$ in \eqref{esta}, near $b\Omega$. Choose a relatively compact subdomain $\Omega_{\eta} \subset\subset\Omega$ so that on $\Omega\setminus\Omega_{\eta}$, the decomposition into tangential and normal part holds, and the condition in \eqref{esta} holds for $(L^J_{u})_{T}$. Then, if we use \eqref{esta} with $L=L_{u}^{J}$, integrate over $\Omega$ and sum over {J}, the error we make (on either side of the inequality) is dominated by $C_{\eta}(\|u\|_{\Omega_{\eta}}^{2}+ \|(L^J_{u})_{N}\|^{2})$. By the estimate on $\|L^J_{u}\|_{1}$ from \cite{Straube10a}, Lemma 2.12 (and the fact that $W^{1}(\Omega)\hookrightarrow L^{2}(\Omega)$ is compact), and by interior elliptic regularity of $\overline{\partial}\oplus\overline{\partial}^{*}$, this error is dominated by $(1-\eta)(\|\overline{\partial}u\|^{2}+\|\overline{\partial}^{*}u\|^{2})+C_{\eta}\|u\|_{-1}^{2}$. For the term coming from the Hessian of $\rho$, we use 
that this Hessian (the Levi form) acts like a subelliptic multiplier (\cite{Kohn79}, Proposition 4.7, part (C), \cite{D'Angelo93}, Proposition 4 in section 6.4.2). The contribution from this term can therefore also be estimated by $(1-\eta)(\|\overline{\partial}u\|^{2}+\|\overline{\partial}^{*}u\|^{2})+C_{\eta}\|u\|_{-1}^{2}$. Putting all of this together, we obtain
\begin{multline}\label{alpha6}
\sideset{}{'}\sum_{J}\int_{\Omega}|\alpha_{\eta}(L_{u}^{J})|^{2} \lesssim (1-\eta)\left(\sideset{}{'}\sum_{J}\int_{\Omega}\sum_{j,k=1}^{n}\frac{\partial^{2}(-h_{\eta})}{\partial z_{j}\partial\overline{z_{k}}}u_{jJ}\overline{u_{kJ}} + \|\overline{\partial}u\|^{2}+\|\overline{\partial}^{*}u\|^{2}\right)+C_{\eta}\|u\|_{-1}^{2}\;;\\
  u\in C^{\infty}_{(0,q)}(\overline{\Omega})\cap dom(\overline{\partial}^{*})\;.
\end{multline}

\smallskip

So far in this section, the assumption on comparable $q$--sums of eigenvalues of the Levi form has not entered the discussion. It will be needed in Proposition \ref{integrated} which gives the estimate on $\sideset{}{'}\sum_{J}\int_{\Omega}|\alpha_{\eta}(L_{u}^{J})|^{2}$ needed to prove Theorem \ref{main}.
\begin{proposition}\label{integrated}
 Assumptions as in Theorem \ref{main}, $q_{0}\leq q\leq (n-1)$. There are a constant $C$ and, for $(1-\eta)$ small enough, a constant $C_{\eta}$, such that 
 \begin{multline}\label{integralest}
 \sideset{}{'}\sum_{J}\int_{\Omega}|\alpha_{\eta}(L_{u}^{J})|^{2} \leq C(1-\eta)(\|\overline{\partial}u\|^{2}+\|\overline{\partial}^{*}u\|^{2}) + C_{\eta}\|u\|_{-1}^{2}\;; \\
 u \in C^{\infty}_{(0,q)}(\overline{\Omega})\cap dom(\overline{\partial}^{*})\;.
 \end{multline}
\end{proposition}
\begin{proof} The Kohn--Morrey--H\"{o}rmander formula gives for a $(0,q)$--form $w\in C^{\infty}_{(0,q)}(\overline{\Omega})\cap dom(\overline{\partial}^{*})$
\begin{equation}\label{KMH}
 \sideset{}{'}\sum_{|J|=(q-1)}\int_{\Omega}\sum_{j,k}\frac{\partial^{2}(-h_{\eta})}{\partial z_{j}\partial\overline{z_{k}}}w_{jJ}\overline{w_{kJ}}e^{h_{\eta}} \leq \|\overline{\partial}w\|_{-h_{\eta}}^{2} + \|\overline{\partial}^{*}_{-h_{\eta}}w\|_{-h_{\eta}}^{2} \;.
\end{equation}
The idea is now to set $w=e^{-(h_{\eta}/2)}u$ (in view of what is needed in \eqref{alpha6}). The $\overline{\partial}^{*}$--term on the right in \eqref{KMH} then simplifies to $\overline{\partial}^{*}u + (1/2)\sideset{}{'}\sum_{J}\left(\sum_{j}(\partial h_{\eta}/\partial z_{j})u_{jJ}\right)d\overline{z_{J}}$ (see the computation on page 116 in \cite{Straube10a}, where this observation is also used). The term $\sideset{}{'}\sum_{J}\left(\sum_{j}(\partial h_{\eta}/\partial z_{j})u_{jJ}\right)d\overline{z_{J}}$ can be handled via $dh_{\eta}=(\alpha_{\eta}-\alpha)$ on $T^{1,0}(b \Omega)$. However, handling the  $\overline{\partial}$--term on the right hand side of \eqref{KMH} requires a modification of this idea.
Set $w:=e^{-\widetilde{(h_{\eta}}/2)}u$ instead, where $\widetilde{h_{\eta}}$ agrees with $h_{\eta}$ on the boundary, and is extended in such a way that $\overline{L_{n}}\widetilde{h_{\eta}}=0$ on the boundary. This amounts to adding to $h_{\eta}$ a function that vanishes on the boundary and has prescribed normal derivative ($\widetilde{h_{\eta}}$ need no longer be real away from the boundary).
Such a function may be chosen to have absolute value not exceeding $1$ (for example), so that $|h_{\eta}-\widetilde{h_{\eta}}|\leq 1$. This definition will have the effect that the normal component of $\overline{\partial}\widetilde{h_{\eta}}$ vanishes on the boundary. Estimate \eqref{KMH} becomes
\begin{equation}\label{KMH2}
 \sideset{}{'}\sum_{J}\int_{\Omega}\sum_{j,k}\frac{\partial^{2}(-h_{\eta})}{\partial z_{j}\partial\overline{z_{k}}}u_{jJ}\overline{u_{kJ}}e^{(h_{\eta}-\Re\widetilde{h_{\eta}})} \leq \|\overline{\partial}(e^{-(\widetilde{h_{\eta}}/2)}u)\|_{-h_{\eta}}^{2} + \|\overline{\partial}^{*}_{-h_{\eta}}(e^{-(\widetilde{h_{\eta}}/2)}u)\|_{-h_{\eta}}^{2} \;,
\end{equation}
 where $\Re\widetilde{h_{\eta}}$ denotes the real part of $\widetilde{h_{\eta}}$. We have
\begin{multline}\label{est4}
 \|\overline{\partial}(e^{-(\widetilde{h_{\eta}}/2)}u)\|_{-h_{\eta}}^{2} = \|e^{(h_{\eta}/2)}\overline{\partial}(e^{-(\widetilde{h_{\eta}}/2)}u)\|^{2} \\ 
 =\|e^{(h_{\eta}-\widetilde{h_{\eta}})/2}(-\frac{1}{2}\overline{\partial}\widetilde{h_{\eta}}\wedge u+\overline{\partial}u)\|^{2}\lesssim \|\overline{\partial}\widetilde{h_{\eta}}\wedge u\|^{2}+\|\overline{\partial}u\|^{2}\;;
\end{multline}
the last inequality results from the estimate $|h_{\eta}-\widetilde{h_{\eta}}|\leq 1$. It is in controlling the wedge term that we will need the normal component of $\overline{\partial}\widetilde{h_{\eta}}$ to vanish on the boundary. The modification also introduces a (benign) error in the $\overline{\partial}^{*}$--term in \eqref{KMH2}:
\begin{multline}\label{est5}
 \|\overline{\partial}^{*}_{-h_{\eta}}(e^{-(\widetilde{h_{\eta}}/2)}u)\|_{-h_{\eta}}^{2}=\|e^{(h_{\eta}/2)}\overline{\partial}^{*}_{-h_{\eta}}(e^{-(\widetilde{h_{\eta}}/2)}u)\|^{2}  \\
 = \|e^{(h_{\eta}-\widetilde{h_{\eta}})/2}(\overline{\partial}^{*}u-\sideset{}{'}\sum_{J}(\sum_{j=1}^{n}\frac{\partial}{\partial z_{j}}(h_{\eta}-\frac{\widetilde{h_{\eta}}}{2})u_{jJ})d\overline{z_{J}})\|^{2} \;\;\;\;\;\;\;\;\;\;\;\;\;\;\;\;\;\;\;\;\;\;\; \\
 \lesssim\|\overline{\partial}^{*}u\|^{2}+\|\sideset{}{'}\sum_{J}(\sum_{j=1}^{n}\frac{\partial h_{\eta}}{\partial z_{j}}u_{jJ})d\overline{z_{J}}\|^{2}+C_{\eta}(\|\rho u\|^{2}+\|u_{N}\|^{2})\;,
\end{multline}
where $u_{N}$ denotes the normal component of $u$. We have used that the tangential components of $\partial(h_{\eta}-\widetilde{h_{\eta}}/2)$ agree with those of $\partial(h_{\eta}/2)$ at the boundary. Note that $\sum_{j=1}^{n}\frac{\partial h_{\eta}}{\partial z_{j}}u_{jJ}=\partial h_{\eta}(L_{u}^{J})=d h_{\eta}(L_{u}^{J})=(\alpha_{\eta}-\alpha)(L_{u}^{J})$ at the boundary. Therefore, $\|\sum_{j=1}^{n}\frac{\partial h_{\eta}}{\partial z_{j}}u_{jJ}\|^{2}\lesssim\|\alpha_{\eta}(L_{u}^{J})\|^{2}+\|\alpha(L_{u}^{J})\|^{2}+ C_{\eta}\|\rho u\|^{2}$. Combining \eqref{KMH2}--\eqref{est5} and collecting the error terms, we arrive at
\begin{multline}\label{est6}
\sideset{}{'}\sum_{J} \int_{\Omega}\sum_{j,k}\frac{\partial^{2}(-h_{\eta})}{\partial z_{j}\partial\overline{z_{k}}}u_{jJ}\overline{u_{kJ}} \\
\lesssim  \sideset{}{'}\sum_{J}\int_{\Omega}\sum_{j,k}\frac{\partial^{2}(-h_{\eta})}{\partial z_{j}\partial\overline{z_{k}}}u_{jJ}\overline{u_{kJ}}e^{(h_{\eta}-\Re\widetilde{h_{\eta}})} + C_{\eta}\|\rho u\|^{2}\;\;\;\;\;\;\;\;\;\;\;\;\;\;\;\;\;\;\;\;  \\
\lesssim \|\overline{\partial}\widetilde{h_{\eta}}\wedge u\|^{2}+\|\overline{\partial}u\|^{2}+\|\overline{\partial}^{*}u\|^{2}\;\;\;\;\;\;\;\;\;\;\;\;\;\;\;\;\;\;\;\;\;\;\; \;\;\;\;\\
 \;\;\;\;\;\;\;\;\;\;\;\;\;\;\;\;\;\;\;\;+\sideset{}{'}\sum_{J}\|\alpha_{\eta}(L_{u}^{J})\|^{2}+C_{\eta}\left(\|\rho u\|^{2}+\|+\|u_{N}\|^{2}\right)\;.
\end{multline}
In the first inequality, we have used that $h_{\eta}$ and $\widetilde{h_{\eta}}$ agree on the boundary; in addition, we have estimated $\|\alpha(L_{u}^{J})\|^{2}$ by $\|u\|^{2}\lesssim \|\overline{\partial}u\|^{2}+\|\overline{\partial}^{*}u\|^{2}$. Finally, using \eqref{alpha6} and summing over $J$ to estimate $\sideset{}{'}\sum_{J}\|\alpha_{\eta}(L_{u}^{J})\|^{2}$, and choosing $(1-\eta)$ so small that the Hessian term can be absorbed into the left hand side of \eqref{est6} gives
\begin{multline}\label{est7}
  \sideset{}{'}\sum_{J}\int_{\Omega}\sum_{j,k}\frac{\partial^{2}(-h_{\eta})}{\partial z_{j}\partial\overline{z_{k}}}u_{jJ}\overline{u_{kJ}}\, \\
  \;\;\;\;\;\;\;\;\;\;\;\;\;\;\;\;\;\;\lesssim \,\|\overline{\partial}\widetilde{h_{\eta}}\wedge u\|^{2}+\|\overline{\partial}u\|^{2}+\|\overline{\partial}^{*}u\|^{2} 
 +C_{\eta}\left(\|\rho u\|^{2}+\|u_{N}\|^{2}+\|u\|_{-1}^{2}\right)\;\;\;\;\;\;\;\;\;\;\;\;\;\;\;\;\;\;\;\;\\
 \lesssim \|\overline{\partial}\widetilde{h_{\eta}}\wedge u\|^{2}+\|\overline{\partial}u\|^{2}+
 \|\overline{\partial}^{*}u\|^{2}+C_{\eta}\|u\|_{-1}^{2}\;.\;\;\;\;\;\;\;\;\;\;
 \end{multline}
In the last inequality, we have used that $\rho$ is a subelliptic multiplier, and the subelliptic estimate for the normal component of a form (\cite{Straube10a}, Lemma 2.12).

\smallskip

To estimate $\|\overline{\partial}\widetilde{h_{\eta}}\wedge u\|^{2}$, we temporarily switch notation, and also express forms in special boundary charts; let $v=\sideset{}{'}\sum_{|K|=q}v_{K}\;\overline{\omega_{K}}$ (supported in a special boundary chart). Then
\begin{equation}\label{est7a}
 |\overline{\partial}\widetilde{h_{\eta}}\wedge v|^{2}\leq|\overline{\partial}\widetilde{h_{\eta}}|^{2}|v|^{2} \lesssim
 \sideset{}{'}\sum_{j<n,K}|(\overline{L_{j}}\widetilde{h_{\eta}})v_{K}|^{2} + \sideset{}{'}\sum_{K}|(\overline{L_{n}}\widetilde{h_{\eta}})v_{K}|^{2}\;.
\end{equation}
Because $\overline{L_{n}}\widetilde{h_{\eta}}=0$ on the boundary, the last term is $O_{\eta}(\rho|v|^{2})$, and so will be under control. For the others, note that $|\overline{L_{j}}\widetilde{h_{\eta}}|^{2}=|\overline{L_{j}}h_{\eta}|^{2}=|\overline{L_{j}h_{\eta}}|^{2}=|L_{j}h_{\eta}|^{2}$ on the boundary when $j<n$. Therefore, when $j<n$,
\begin{multline}\label{est8}
 |(\overline{L_{j}}\widetilde{h_{\eta}})v_{K}|^{2} \lesssim |(L_{j}h_{\eta})v_{K}|^{2}+C_{\eta}|\rho v|^{2}  \\
 =|\partial h_{\eta}(v_{K}L_{j})|^{2}+C_{\eta}|\rho v|^{2}=|(\alpha_{\eta}-\alpha)(v_{K}L_{j})|^{2}+C_{\eta}|\rho v|^{2}  \\
 \lesssim |\alpha_{\eta}(v_{K}L_{j})|^{2}+|v|^{2}+C_{\eta}|\rho v|^{2}\;.
\end{multline}
The error that results from the fact that $\partial h_{\eta}=\alpha_{\eta}-\alpha$ only at the boundary is also covered by $C_{\eta}|\rho v|^{2}$. Integrating over $\Omega$, invoking the case $q=1$ of \eqref{esta} and \eqref{alpha6} with $u=v_{K}\overline{\omega_{j}}$ (so $L_{u}=v_{K}L_{j}$: for a $(0,1)$--form $u$, $L_{u}$ is characterized by $\overline{u}(L_{u})=|u|^{2}$ and $\{\overline{u}\}^{\perp}$ annihilates $L_{u}$), and then \eqref{est7}, we get 
\begin{multline}\label{est9}
 \int_{\Omega} |\overline{L_{j}}\widetilde{h_{\eta}}v_{K}|^{2} \lesssim \int_{\Omega}|\alpha_{\eta}(v_{K}L_{j})|^{2}+\|v\|^{2}+C_{\eta}\|\rho v\|^{2} \\
 \lesssim (1-\eta)\left(\|\overline{\partial}\widetilde{h_{\eta}}\wedge(v_{K}\overline{\omega_{j}})\|^{2}+\|\overline{\partial}(v_{K}\overline{\omega_{j}})\|^{2}+\|\overline{\partial}^{*}(v_{K}\overline{\omega_{j}})\|^{2}\right)  \\
 +\|v\|^{2}+C_{\eta}\left(\|\rho v\|^{2}+\|v\|_{-1}^{2}\right)\;.
\end{multline}
Because $q_{0}\leq q$, and in view of Lemma \ref{percup}, the comparable $q_{0}$--sums assumption on the Levi eigenvalues in Theorem \ref{main} implies the maximal estimate \eqref{max} for $(0,q)$--forms. Accordingly 
\begin{equation}\label{est10}
 \|\overline{\partial}(v_{K}\overline{\omega_{j}})\|^{2}+\|\overline{\partial}^{*}(v_{K}\overline{\omega_{j}})\|^{2} \lesssim \sum_{s}\|\overline{L_{s}}v_{K}\|^{2}+\|L_{j}v_{K}\|^{2}+\|v_{K}\|^{2} \lesssim \|\overline{\partial}v\|^{2}+\|\overline{\partial}^{*}v\|^{2}\;,
\end{equation}
where the second inequality results from the maximal estimates applied to the $(0,q)$--form $v$ (note that $j<n$). Starting from \eqref{est7a}, inserting \eqref{est10} into \eqref{est9}, and using  $|\overline{\partial}\widetilde{h_{\eta}}\wedge (v_{K}\overline{\omega_{j}})|^{2}\lesssim |\overline{\partial}\widetilde{h_{\eta}}|^{2}|v|^{2}$, we have
\begin{multline}\label{est11}
 \int_{\Omega}|\overline{\partial}\widetilde{h_{\eta}}|^{2}|v|^{2} \leq
 \sideset{}{'}\sum_{j<n,K}\int_{\Omega}|(\overline{L_{j}}\widetilde{h_{\eta}})v_{K}|^{2}+C_{\eta}\|\rho v\|^{2}\\
 \lesssim (1-\eta)\left(\int_{\Omega}|\overline{\partial}\widetilde{h_{\eta}}|^{2}|v|^{2}+\|\overline{\partial}v\|^{2}+\|\overline{\partial}^{*}v\|^{2}\right)  
 +\|v\|^{2}+C_{\eta}\|v\|_{-1}^{2}\;.
\end{multline}
 Now we choose $(1-\eta)$ small enough so that the term containing $\overline{\partial}\widetilde{h_{\eta}}$ in \eqref{est11} can be absorbed into the left hand side. Then
\begin{multline}\label{est7b}
 \|\overline{\partial}\widetilde{h_{\eta}}\wedge v\|^{2} \leq \int_{\Omega}|\overline{\partial}\widetilde{h_{\eta}}|^{2}|v|^{2}\lesssim (1-\eta)\left(\|\overline{\partial}v\|^{2}+\|\overline{\partial}^{*}v\|^{2}\right) 
 +\|v\|^{2}+C_{\eta}\|v\|_{-1}^{2}\;.
\end{multline}
Via a partition of unity, \eqref{est7b} carries over as usual to when $v$ is not supported in a special boundary chart; the compactly supported term that arises is also dominated by the right hand side of \eqref{est7b} (by interior elliptic regularity). Insert this estimate into \eqref{est7}. The result is the first inequality below
\begin{equation}\label{est12}
 \sideset{}{'}\sum_{J}\int_{\Omega}\sum_{j,k}\frac{\partial^{2}(-h_{\eta})}{\partial z_{j}\partial\overline{z_{k}}}u_{jJ}\overline{u_{kJ}}\,\lesssim\,\|\overline{\partial}u\|^{2}+\|\overline{\partial}^{*}u\|^{2}+C_{\eta}\|u\|_{-1}^{2} \;.
\end{equation}
 Inserting \eqref{est12} into \eqref{alpha6} gives \eqref{integralest}. The proof of Proposition \ref{integrated} is now complete.
 \end{proof}

\section{Proof of Theorem \ref{main}}\label{proof}

\begin{proof}[Proof of Theorem \ref{main}]
It suffices to prove the statement for the $\overline{\partial}$--Neumann operator (\cite{BoasStraube90}; \cite{Straube10a}, Theorem 5.5). The proof follows the first part of the proof of Theorem 1 in section 4 of \cite{Straube05}, but then uses the commutator formulas from \cite{HarringtonLiu20} (i.e. Lemmas \ref{dbar} and \ref{dbar*} above) instead of Lemmas 4 and 5 in \cite{Straube05}. 

\smallskip

We use a downward induction on the degree $q$. In the top degree $q=n$, the $\overline{\partial}$--Neumann boundary conditions reduce to Dirichlet boundary conditions, and $N_{n}$ gains two derivatives in Sobolev norms. So to show that $N_{q_{0}}$ satisfies Sobolev estimates, it suffices to show: if $N_{q+1}$ satisfies Sobolev estimates, and $q\geq q_{0}$, then so does $N_{q}$. So fix such a $q$. The induction assumption will be invoked to conclude that the Bergman projection $P_{q}$ on $(0,q)$--forms satisfies Sobolev estimates (\cite{Straube10a}, Theorem 5.5). 

\smallskip

As usual, the arguments will involve absorbing terms, and so one has to know that these terms are finite. In order to insure that, we work with the regularized operators $N_{\delta,q}$ that result from elliptic regularization. That is, $N_{\delta,q}$ is the inverse of the operator $\Box_{\delta,q}$ associated with the quadratic form $Q_{\delta,q}(u,u)=\|\overline{\partial}u\|^{2}+\|\overline{\partial}^{*}u\|^{2}+\delta\|\nabla u\|^{2}$, with form domain $W^{1}_{(0,q)}(\Omega)\cap dom(\overline{\partial}^{*})$, where $\nabla$ denotes the vector of all first order derivatives of all coefficients of $u$. $N_{\delta,q}$ maps $L^{2}_{(0,q)}(\Omega)$ continuously into the form domain, endowed with the norm $Q_{\delta,q}(u,u)^{1/2}$. We will prove estimates with constants that are uniform in $\delta$; letting $\delta\rightarrow 0$ then gives the desired estimates for $N_{q}$ (\cite{ChenShaw01}, pages 102--103 and \cite{Straube10a}, section 3.3).

\smallskip
 
So $q$ is now fixed, $q_{0}\leq q$, and $N_{q+1}$, hence $P_{q}$, satisfy Sobolev estimates. We want to prove Sobolev estimates for $N_{\delta,q}$, with constants uniform in $\delta$ ($\delta$ small). Because $P_{q}$ is regular in Sobolev norms, Lemma 7 in \cite{Straube05} (see also \cite{Straube25}, Theorem 1.1)  says that for all $s>0$, there is a constant $C_{s}$ independent of $\delta$ such that we have the estimate 
\begin{equation}\label{est34}
 \|N_{\delta,q}u\|_{s} \leq C_{s}(\|\overline{\partial}N_{\delta,q}u\|_{s} + \|\overline{\partial}^{*}N_{\delta,q}u\|_{s})\;.
\end{equation}
Therefore, it suffices to prove estimates for the right hand side of \eqref{est34} (with constants independent of $\delta$). In order to bring $Q_{\delta,q}$ into play, we prove estimates for $\overline{\partial}N_{\delta,q}$, $\overline{\partial}^{*}N_{\delta,q}$, and $\delta^{1/2}\nabla N_{\delta,q}$ as in \cite{Straube05}. We do this by induction on the Sobolev index $k$.
 
\smallskip

Because $N_{\delta,q}$ maps $L^{2}_{(0,q)}(\Omega)$ continuously into the form domain of $Q_{\delta,q}$, with norm less than or equal to $1$, the following inductive assumption holds for $l=0$: there are $\eta\in (0,1)$, a constant $C$ and $\delta_{0}>0$  such that
\begin{equation}\label{est30}
 \|\overline{\partial}N_{\delta,q}u\|_{l,2h_{\eta}}^{2}+\|\overline{\partial}^{*}N_{\delta,q}u\|_{l,2h_{\eta}}^{2}+\delta\|\nabla N_{\delta,q}u\|_{l,2h_{\eta}}^{2} \leq C\|u\|_{l}^{2}\;,\;0<\delta\leq\delta_{0}\;.
\end{equation}
Indeed, take any $\eta\in (0,1)$, then choose $C$ big enough and $\delta>0$ arbitrary. We now assume this induction assumption holds for $0\leq l\leq (k-1)$ and show that it then holds for $l=k$. The induction assumption \eqref{est30} corresponds to (31) in \cite{Straube05}, with the modification that the norms on the left hand side are weighted, and that there is no term $\delta^{2}\|u\|_{l+1}^{2}$. It turns out that with Lemmas \ref{dbar} and \ref{dbar*} for the commutators with $\overline{\partial}$ and $\overline{\partial}^{*}$ (rather than Lemmas 4 and 5 from \cite{Straube05}), this term is not needed for the induction to run.

\smallskip

For $\eta\in (0,1)$, set $X_{\eta}=e^{\rho_{\eta} g_{\eta}}L_{n}^{\rho_{\eta}}$, where $g_{\eta}$ is a smooth function that agrees on the boundary with $\alpha_{\eta}(\overline{L_{n}^{\rho_{\eta}}})$. Because $\alpha_{\eta}(T^{\rho_{\eta}})=0$, $\alpha_{\eta}\big(\overline{L_{n}^{\rho_{\eta}}}\big)$ is real, and we may take $g_{\eta}$ to be real valued. This modification ia analogous to replacing $h_{\eta}$ by $\widetilde{h_{\eta}}$ in section \ref{index-form}. Then $(X_{\eta}-\overline{X_{\eta}})$ is tangential at the boundary. The reason for the modification is that the terms where $j=n$ in the commutator \eqref{est41} below will then vanish on the boundary. This fact is needed in the argument. Being able to make this choice is a manifestation of the principle that `commutator conditions in the normal (and any strictly pseudoconvex) direction come for free' (\cite{Straube10a}, Section 5.7, \cite{BoasStraube91a}, proof of the lemma). Note that because $g_{\eta}$ is only prescribed on the boundary, we may assume that $|\rho_{\eta} g_{\eta}|\leq 1/2k$. As usual, we need tangential differential operators to preserve the domain of $\overline{\partial}^{*}$, and so we let them act in special boundary charts (\cite{Straube10a}, section 2.2). The error this introduces is one order lower than the operator, and as a result, our estimates are not affected.

\smallskip

Using Lemma \ref{benign} for barred derivatives or derivatives that are complex tangential, we have
\begin{multline}\label{est35}
 \|\overline{\partial}^{*}N_{\delta,q}u\|_{k,2h_{\eta}}^{2} \lesssim \|(e^{-h_{\eta}}(L_{n}^{\rho}-\overline{L_{n}^{\rho}}))^{k}\overline{\partial}^{*}N_{\delta,q}u\|^{2}  \\
 +C_{\eta}(\|\overline{\partial}\overline{\partial}^{*}N_{\delta,q}u\|_{k-1}^{2}+\|\overline{\partial}^{*}N_{\delta,q}u\|_{k-1}\|\overline{\partial}^{*}N_{\delta,q}u\|_{k}+\|u\|_{k}^{2})\;.
\end{multline}
The $\|u\|_{k}^{2}$ term is needed because of the compactly supported term from the partition of unity used for the special boundary charts; this term is not covered by Lemma \ref{dbar}. The estimate follows from interior elliptic regularity of $\Box_{\delta,q}$, with constants that are uniform in $\delta$. This uniformity holds because for $u\in dom(\Box_{\delta,q})$, $\Box_{\delta,q}u=-(1/4+\delta)\Delta u$, where $\Delta$ acts coefficient wise (see \cite{Straube10a}, formula (3.22)). We have used that when there is a barred derivative, or a complex tangential derivative, we can always commute it so that it acts first. Then there is no issue with the requirement that a form must be in the domain of $\overline{\partial}^{*}$ in order for Lemma \ref{benign} to apply. The error term this makes is of order at most $(k-1)$. Modulo terms controlled by $\|\overline{\partial}^{*}N_{\delta,q}u\|_{k-1}^{2}$, we have $\|(e^{-h_{\eta}}(L_{n}^{\rho}-\overline{L_{n}^{\rho}}))^{k}\overline{\partial}^{*}N_{\delta,q}u\|^{2} \lesssim\|e^{-kh_{\eta}}(L_{n}^{\rho}-\overline{L_{n}^{\rho}})^{k}\overline{\partial}^{*}N_{\delta,q}u\|^{2}\leq e\|(X_{\eta}-\overline{X_{\eta}})^{k}\overline{\partial}^{*}N_{\delta,q}u\|^{2}$ (since $2k|\rho g_{\eta}|\leq 1$). Upon also estimating the second to last term in \eqref{est35} via a s.c.--l.c argument and absorbing the term $s.c\|\overline{\partial}^{*}N_{\delta,q}u\|_{k,2h_{\eta}}^{2}$ into the left hand side of \eqref{est35}, we thus arrive at
\begin{equation}\label{est36}
 \|\overline{\partial}^{*}N_{\delta,q}u\|_{k,2h_{\eta}}^{2} \lesssim \|(X_{\eta}-\overline{X_{\eta}})^{k}\overline{\partial}^{*}N_{\delta,q}u\|^{2} 
 +C_{\eta}(\|\overline{\partial}\overline{\partial}^{*}N_{\delta,q}u\|_{k-1}^{2}+\|\overline{\partial}^{*}N_{\delta,q}u\|_{k-1}^{2}+\|u\|_{k}^{2})\;.
\end{equation}
This estimate corresponds to estimate (32) in \cite{Straube05}, with $\eta$ in the role of $\varepsilon$. Note that uniform boundedness of $h_{\eta}$, used in \cite{Straube05} for estimate (32), is not needed here because we use the weighted norm on the left hand side.

\smallskip

Applying the same reasoning to $\|\overline{\partial}N_{\delta,q}u\|_{k,2h_{\eta}}^{2}$ takes a little more care, because $\overline{\partial}N_{\delta,q}u$ need not be in the domain of $\overline{\partial}^{*}$ (the free boundary condition imposed by $Q_{\delta}$ is different from that imposed by $Q$). The necessary modification is explained in detail in \cite{Straube05}, to which we refer the reader. The result is the following estimate:
\begin{multline}\label{est37}
 \|\overline{\partial}N_{\delta,q}u\|_{k,2h_{\eta}}^{2} \lesssim \|(X_{\eta}-\overline{X_{\eta}})^{k}\overline{\partial}N_{\delta,q}u\|^{2}   \\
 + C_{\eta}(\|\vartheta\overline{\partial}N_{\delta,q}u\|_{k-1}^{2}+\|\overline{\partial}N_{\delta,q}u\|_{k-1}^{2}+\|\delta(\partial/\partial\nu)N_{\delta,q}u\wedge\overline{\omega_{n}}\|_{k}^{2}+\|u\|_{k}^{2})\;.
\end{multline}
Here, $\vartheta$ is the formal adjoint of $\overline{\partial}$, $(\partial/\partial\nu)$ denotes the normal derivative acting coefficient wise on forms, and $\omega_{n}$ is the form dual to the unit normal $L_{n}$. The passage from $\|(e^{-h_{\eta}}(L_{n}^{\rho}-\overline{L_{n}^{\rho}}))^{k}\overline{\partial}N_{\delta,q}u\|^{2}$ to $\|(X_{\eta}-\overline{X_{\eta}})^{k}\overline{\partial}N_{\delta,q}u\|^{2}$ is as in \eqref{est35} -- \eqref{est36}.

\smallskip

The argument for $\delta\|\nabla N_{\delta,q}u\|_{k,2h_{\eta}}^{2}$ is slightly different from that given in \cite{Straube05} (the argument there uses uniform bounds on the functions $h_{\eta}$). At issue is the $\partial/\partial\nu$ term in $\nabla$, which need not be in the domain of $\overline{\partial}^{*}$. However, because the boundary is not characteristic for $\overline{\partial}\oplus\vartheta$ (\cite{Straube10a}, Lemma 2.2), we can write $\nabla N_{\delta,q}u = \nabla_{T}N_{\delta,q}u$ plus a linear combination of coefficients of $\overline{\partial}N_{\delta,q}u$, $\overline{\partial}^{*}N_{\delta,q}u$, and $N_{\delta,q}u$ itself. Here, $\nabla_{T}$ stands for tangential derivatives. Applying the same reasoning as for \eqref{est36} to $\nabla_{T}N_{\delta,q}u$ gives
\begin{multline}\label{est38}
 \delta\|\nabla N_{\delta,q}u\|_{k,2h_{\eta}}^{2} \lesssim \delta\|(X_{\eta}-\overline{X_{\eta}})^{k}\nabla N_{\delta,q}u\|^{2}  \\
 + C_{\eta}\delta(\|\overline{\partial}N_{\delta,q}u\|_{k}^{2}+\|\overline{\partial}^{*}N_{\delta,q}u\|_{k}^{2} + \|N_{\delta,q}u\|_{k}\;\|N_{\delta,q}u\|_{k+1})\;.
\end{multline}
We have used that $\|(X_{\eta}-\overline{X_{\eta}})^{k}\nabla_{T} N_{\delta,q}u\|^{2} \leq \|(X_{\eta}-\overline{X_{\eta}})^{k}\nabla N_{\delta,q}u\|^{2}$.

\smallskip

Estimates \eqref{est36},\eqref{est37}, and \eqref{est38} correspond to estimates (32), (33), and (34) in \cite{Straube05}. It is shown there in (31)--(39), and the paragraph immediately following (39), that completing the induction step reduces to estimating the following three inner products:
\begin{equation}\label{eq31}
 \left((X_{\eta}-\overline{X_{\eta}})^{k}\overline{\partial}N_{\delta,q}u,[\overline{\partial},X_{\eta}-\overline{X_{\eta}}](X_{\eta}-\overline{X_{\eta}})^{k-1}N_{\delta,q}u\right)\;,
\end{equation}
\begin{equation}\label{eq32}
 \left((X_{\eta}-\overline{X_{\eta}})^{k}\overline{\partial}^{*}N_{\delta,q}u,[\overline{\partial}^{*},X_{\eta}-\overline{X_{\eta}}](X_{\eta}-\overline{X_{\eta}})^{k-1}N_{\delta,q}u\right)\;,
\end{equation}
 and
 \begin{equation}\label{eq33}
  \delta\left((X_{\eta}-\overline{X_{\eta}})^{k}\nabla N_{\delta,q}u,[\nabla,X_{\eta}-\overline{X_{\eta}}](X_{\eta}-\overline{X_{\eta}})^{k-1}N_{\delta,q}u\right)\;.
 \end{equation}
More precisely, it is shown in \cite{Straube05} that for $\eta$ given, there is $\delta_{0}(\eta)$ such that all the error terms in \eqref{est36}--\eqref{est38} are acceptable for \eqref{est30} (for $l=k$), or can be absorbed into the left hand side of \eqref{est30} when $\delta\leq \delta_{0}(\eta)$. We will use `acceptable' somewhat loosely for terms in an estimate that are either bounded appropriately by the right hand side, or can be absorbed into the left hand side. Note that upon adding \eqref{eq31} through \eqref{eq33}, using Cauchy--Schwarz on each inner product, followed by a s.c.--l.c. argument, the squares of the norms of the left hand sides in \eqref{eq31}--\eqref{eq33} can be absorbed into the first term on the right hand side of the sum of \eqref{est36}--\eqref{est38}. Therefore, in order to estimate the sum of the left hand sides of \eqref{est36}--\eqref{est38}, we only have to estimate 
 \begin{multline}\label{39}
  \|[\overline{\partial},X_{\eta}-\overline{X_{\eta}}](X_{\eta}-\overline{X_{\eta}})^{k-1}N_{\delta,q}u\|^{2}+\|[\overline{\partial}^{*},X_{\eta}-\overline{X_{\eta}}](X_{\eta}-\overline{X_{\eta}})^{k-1}N_{\delta,q}u\|^{2}  \\
  + \delta\|[\nabla,X_{\eta}-\overline{X_{\eta}}](X_{\eta}-\overline{X_{\eta}})^{k-1}N_{\delta,q}u\|^{2}\;.
 \end{multline}
 This expression differs from 
 \begin{equation}\label{39a}
 \|[\overline{\partial},(X_{\eta}-\overline{X_{\eta}})^{k}]N_{\delta,q}u\|^{2}+\|[\overline{\partial}^{*},(X_{\eta}-\overline{X_{\eta}})^{k}]N_{\delta,q}u\|^{2}  
  + \delta\|[\nabla,(X_{\eta}-\overline{X_{\eta}})^{k}]N_{\delta,q}u\|^{2}
 \end{equation}
by terms that are of order $(k-1)$. This follows from the formula for commutators with powers of an operator cited earlier: $[A,T^{k}] = \sum_{j=1}^{k}\binom{k}{j}[\cdots[[A,T],T]\cdots]T^{k-j}$ ($j$--fold, note that these iterated commutators are of order one). In view of \eqref{est34} and the induction assumption, these terms are acceptable. We will now estimate \eqref{39a}.

\smallskip

For the commutators with $\overline{\partial}$ and $\overline{\partial}^{*}$ we use Lemmas \ref{dbar} and \ref{dbar*}, with $\rho=\rho_{\eta}$, $h = -\rho_{\eta} g_{\eta}$ and $v=N_{\delta,q}u$. It suffices to estimate the terms in \eqref{39a} with $N_{\delta,q}u$ replaced by $\chi N_{\delta,q}u$, where $\chi\in C^{\infty}(\mathbb{C}^{n})$ is supported in a boundary chart. Via a suitable partition of unity (independent of $\eta$), these estimates will then be summed to obtain global estimates. Lemma \ref{dbar} gives
 \begin{multline}\label{est41}
 \|[\overline{\partial},(X_{\eta}-\overline{X_{\eta}})^{k}]\chi N_{\delta,q}u\|^{2}  \\
 \lesssim \sideset{}{'}\sum_{J}\sum_{j\notin J, j<n}\int_{\Omega}\big|\overline{L_{j}}(\rho_{\eta} g_{\eta})-\alpha_{\eta}(\overline{L_{j}})\big|^{2}|(X_{\eta}-\overline{X_{\eta}})^{k}(\chi N_{\delta,q}u)_{J}|^{2}\;\;\;\;\;\;\;\;\;\;\;\; \\
 +\sideset{}{'}\sum_{n\notin J}\int_{\Omega}\big|\overline{L_{n}^{\rho_{\eta}}}(\rho_{\eta} g_{\eta})-\alpha_{\eta}(\overline{L_{n}^{\rho_{\eta}}})\big|^{2}|(X_{\eta}-\overline{X_{\eta}})^{k}(\chi N_{\delta,q}u)_{J}|^{2}\;\;\;\;\;\;\; \\
 \;\;\;\;+C_{\eta}\,(\|\overline{\partial}N_{\delta,q}u\|_{k-1}^{2}+\|\overline{\partial}^{*}N_{\delta,q}u\|_{k-1}^{2}+\|N_{\delta,q}u\|_{k-1}\|N_{\delta,q}u\|_{k}+\|u\|_{k-2}^{2})\;.
 \end{multline}
We have used \eqref{est34} to estimate the terms of order at most $(k-1)$. The sum is only over $j\notin J$ because for $j\in J$, the term $\overline{\omega_{j}}\wedge\overline{\omega_{J}}$ vanishes. In the error terms, we omit the factor $\chi$ and use that the commutators of $\chi$ with $\overline{\partial}$ and $\overline{\partial}^{*}$ are of order zero. Because there are only finitely many $\chi$s, we may keep $C_{\eta}$ independent of $\chi$. The term $\|u\|_{k-2}^{2}$ again results from the compactly supported term that arises from letting $(X_{\eta}-\overline{X_{\eta}})$ act in special boundary charts; as above, the estimate follows from interior elliptic regularity of $\Box_{\delta,q}$, with constants that are uniform in $\delta$. The order $(k-1)$ error terms are acceptable by the induction hypothesis, and using \eqref{est34} for $\|N_{\delta,q}u\|_{k-1}\|N_{\delta,q}u\|_{k}$ shows that this term can be split into an absorbable term and a lower order term. From Lemma \ref{dbar*}, we similarly obtain (but not taking the absolute values all the way inside the sums)
\begin{multline}\label{est42}
 \|[\overline{\partial}^{*},(X_{\eta}-\overline{X_{\eta}})^{k}]\chi N_{\delta,q}u\|^{2}  \\
 \lesssim \sideset{}{'}\sum_{S}\int_{\Omega}\big|\sum_{j}\big(L_{j}(\rho_{\eta} g_{\eta})-\alpha_{\eta}(L_{j})\big)(X_{\eta}-\overline{X_{\eta}})^{k}(\chi N_{\delta,q}u)_{jS}\big|^{2}\;\;\;\;\;\;\;\;\;\;\;\;\;\;\;\;\;\;\;\;\;\;\;\;\;\;\;\;\;\;\;\;\;\;\;\;\; \\
 \;\;\;\;\;\;\;\;\;\;\;\;\;\;+ C_{\eta}\,\left(\|\overline{\partial}N_{\delta,q}u\|_{k-1}^{2}+\|\overline{\partial}^{*}N_{\delta,q}u\|_{k-1}^{2}+\|N_{\delta,q}u\|_{k-1}\|N_{\delta,q}u\|_{k}+\|u\|_{k-2}^{2}\right)\;.
\end{multline} 
The first three error terms in the last line are the same as in \eqref{est41} and so are acceptable, as is the last. For the fourth one, we use the subelliptic estimate $\|(N_{\delta,q})_{N}\|_{k}^{2} \lesssim \|\overline{\partial}N_{\delta,q}u\|_{k-1}^{2}+\|\overline{\partial}^{*}N_{\delta,q}u\|_{k-1}^{2}+\|N_{\delta,q}u\|_{k-1}^{2}$ (\cite{Straube10a}, estimate (2.96)). Using \eqref{est34} again and the induction hypothesis shows that this term is also acceptable.
 
For the commutator with $\nabla$ in \eqref{39a}, we do not need the cutoff functions. Note that the term is of order $k$, so that the term is estimated by $\delta C_{\eta}\|N_{\delta,q}u\|_{k}^{2}\lesssim \delta C_{\eta}(\|\overline{\partial}N_{\delta,q}u\|_{k,2h_{\eta}}^{2}+\|\overline{\partial}^{*}N_{\delta,q}u\|_{k,2h_{\eta}}^{2})$ (again by \eqref{est34}, the weight in the norms just changes the constant $C_{\eta}$). If we choose $\delta_{0}(\eta)$ small enough, this term can therefore be absorbed when $\delta\leq \delta_{0}(\eta)$. We are therefore left with estimating only the main terms on the right hand sides of \eqref{est41} and \eqref{est42}, respectively.

For $j<n$, the contributions in \eqref{est41} and \eqref{est42} coming from $\overline{L_{j}}(\rho_{\eta} g_{\eta})$ and $L_{j}(\rho_{\eta} g_{\eta})$, respectively, are $\mathcal{O}_{\eta}(\|\rho(X_{\eta}-\overline{X_{\eta}})^{k} \chi N_{\delta,q}u\|_{k}$ and are again acceptable ($\rho_{\eta}$ is a subelliptic multiplier and \eqref{est34}). For the term in \eqref{est41} where $j=n$, we have that $\big|\overline{L_{n}^{\rho_{\eta}}}(\rho_{\eta} g_{\eta})-\alpha_{\eta}(\overline{L_{n}^{\rho_{\eta}}})|^{2}=0$ on the boundary, due to the choice of $g_{\eta}$ (note that $\overline{L_{n}^{\rho_{\eta}}}(\rho_{\eta} g_{\eta})=g_{\eta}$ on the boundary). The contribution from this term is thus $\mathcal{O}_{\eta}(\|\rho(X_{\eta}-\overline{X_{\eta}})^{k}(\chi N_{\delta,q}u)\|^{2})$. Again because $\rho_{\eta}$ is a subelliptic multiplier and \eqref{est34}, this term is acceptable. In \eqref{est42}, $L_{n}(\rho_{\eta}g_{\eta})$ multiplies a coefficient of the normal component of $(X_{\eta}-\overline{X_{\eta}})\chi N_{\delta,q}u$. Taking the normal component acts like a subelliptic multiplier (\cite{Straube10a}, estimate (2.96)), so this term is also benign.

We now discuss the remaining terms in the second lines of \eqref{est41} and \eqref{est42}, respectively. We want to apply Proposition \ref{integrated}. In order to do so, observe that it does not matter whether we form the vector fields $L^{J}_{u}$ in Euclidean coordinates or with coordinates in a special boundary frame: one set is obtained from the other via linear combinations (with smooth coefficients) when the from $u$ is fixed. Therefore, the validity of \eqref{integralest} is not affected. We start with \eqref{est41}; recall that for $j<n$, the terms containing $g_{\eta}$ have been taken care of, as has the term where $j=n$. Also, if $n\in J$, we can use the subelliptic estimate for the normal component of a form once more; these terms are thus acceptable. Because $\alpha_{\eta}$ is real, $|\alpha_{\eta}(\overline{L_{j}})|^{2} = |\alpha_{\eta}(L_{j})|^{2}$, and we only have to estimate $\int_{\Omega}|\alpha_{\eta}(L_{j})|^{2}|(X_{\eta}-\overline{X_{\eta}})^{k}(\chi N_{\delta,q}u)_{J}|^{2}=\int_{\Omega}|\alpha_{\eta}((X_{\eta}-\overline{X_{\eta}})^{k}(\chi N_{\delta,q}u)_{J}L_{j})|^{2}$. Because $j\notin J$, $L_{\overline{\omega_{j}}\wedge\overline{\omega_{J}}}^{J} = \sum_{s=1}^{n}(\overline{\omega_{j}}\wedge\overline{\omega_{J}})_{sJ}L_{s} = L_{j}$ (the sum is only over $s\notin J$), and so $L^{J}_{\big((X_{\eta}-\overline{X_{\eta}})^{k}\chi N_{\delta,q}u\big)\overline{\omega_{j}}\wedge\overline{\omega_{J}}}=\big((X_{\eta}-\overline{X_{\eta}})^{k}\chi N_{\delta,q}u\big)_{J}L_{j}$. Proposition \ref{integrated} applies to the $(0,q+1)$--form $\big((X_{\eta}-\overline{X_{\eta}})^{k}\chi N_{\delta,q}u\big)_{J}(\overline{\omega_{j}}\wedge\overline{\omega_{J}})$; note that because $j<n$ and $n\notin J$, this form is in the domain of $\overline{\partial}^{*}$. The resulting estimate is
\begin{multline}\label{est43}
 \sideset{}{'}\sum_{J}\sum_{j\notin J,j<n}\int_{\Omega}|\alpha_{\eta}(L_{j})|^{2}|(X_{\eta}-\overline{X_{\eta}})^{k}(\chi N_{\delta,q}u)_{J}|^{2}  \\
 \lesssim (1-\eta)\sideset{}{'}\sum_{J}\sum_{j\notin J,j<n}\left(\|\overline{\partial}\big(((X_{\eta}-\overline{X_{\eta}})^{k}(\chi N_{\delta,q}u)_{J})(\overline{\omega_{j}}\wedge\overline{\omega_{J}})\big)\|^{2}\right.\\
 \;\;\;\;\;\;\;\;\;\;\;\;\;\;\;\;\;\;\;\;\;\;\;\;\;\;\;\;\;+\left.\|\overline{\partial}^{*}\big(((X_{\eta}-\overline{X_{\eta}})^{k}(\chi N_{\delta,q}u)_{J})(\overline{\omega_{j}}\wedge\overline{\omega_{J}})\big)\|^{2}\right) \\
 +C_{\eta}\|(X_{\eta}-\overline{X_{\eta}})^{k}N_{\delta,q}u\|_{-1}^{2}\;.\;\;\;
\end{multline}
We have again omitted the factor $\chi$ from the error term; commuting $\chi$ to the front gives an error term that is of the order of $C_{\eta}\|N_{\delta,q}u\|_{k-1}$. These terms are acceptable. Once $\chi$ is first, the terms are bounded by the norms over all of $\Omega$. The $\overline{\partial}$ and $\overline{\partial}^{*}$ terms in the second and third lines of \eqref{est43} are dominated (uniformly in $\eta$ and $\delta$) by 
\begin{multline}\label{est44}
 \sum_{s=1}^{n}\|\overline{L_{s}}((X_{\eta}-\overline{X_{\eta}})^{k}(\chi N_{\delta,q}u)_{J})\|^{2}+\sum_{j<n}\|L_{j}((X_{\eta}-\overline{X_{\eta}})^{k}(\chi N_{\delta,q}u)_{J})\|^{2}+\|(X_{\eta}-\overline{X_{\eta}})^{k}(\chi N_{\delta,q}u)\|^{2}  \\
 \lesssim \|\overline{\partial}((X_{\eta}-\overline{X_{\eta}})^{k}(\chi N_{\delta,q}u))\|^{2}+\|\overline{\partial}^{*}((X_{\eta}-\overline{X_{\eta}})^{k}(\chi N_{\delta,q}u))\|^{2}\;;
\end{multline}
we have used Lemma \ref{benign} for the first term, maximal estimates for the second (see \eqref{max}; note again that $q_{0}\leq q$, and Lemma \ref{percup}), and the basic $L^{2}$ estimate for the third. There is no $L_{n}$--term because the normal component of $\overline{\omega_{j}}\wedge\overline{\omega_{J}}$ is zero (as $j<n$, $n\notin J$).

\smallskip

In \eqref{est42}, the integrand equals $\big|\alpha_{\eta}\big(\sum_{j}(X_{\eta}-\overline{X_{\eta}})^{k}(\chi N_{\delta,q}u)_{jS}L_{j}\big)\big|^{2}$ (again omitting the contribution from the $g_{\eta}$--term). The vector field inside $\alpha_{\eta}(\cdot)$ equals $L^{S}_{(X_{\eta}-\overline{X_{\eta}})^{k}(\chi N_{\delta,q}u)}$. The form $(X_{\eta}-\overline{X_{\eta}})^{k}(\chi N_{\delta,q}u)$ is in the domain of $\overline{\partial}^{*}$, and Proposition \ref{integrated} applies and gives
\begin{multline}\label{nice}
 \sideset{}{'}\sum_{S}\int_{\Omega}\big|\sum_{j<n}\alpha_{\eta}(L_{j})(X_{\eta}-\overline{X_{\eta}})^{k}(\chi N_{\delta,q}u)_{jS}\big|^{2}\\
 \lesssim (1-\eta)\left(\|\overline{\partial}\big((X_{\eta}-\overline{X_{\eta}})^{k}(\chi N_{\delta,q}u)\big)\|^{2}+\|\overline{\partial}^{*}\big((X_{\eta}-\overline{X_{\eta}})^{k}(\chi N_{\delta,q}u)\big)\|^{2}\right)\\
 + C_{\eta}\|N_{\delta,q}u\|_{k-1}^{2}\;.
\end{multline}

\smallskip

Combining \eqref{est41}--\eqref{est44}, and the remark above about the commutator with $\nabla$ in \eqref{39a} and summing over the partition of unity gives
\begin{multline}\label{est45}
 \|[\overline{\partial},(X_{\eta}-\overline{X_{\eta}})^{k}]N_{\delta,q}u\|^{2}+\|[\overline{\partial}^{*},(X_{\eta}-\overline{X_{\eta}})^{k}]N_{\delta,q}u\|^{2}  
  + \delta\|[\nabla,(X_{\eta}-\overline{X_{\eta}})^{k}]N_{\delta,q}u\|^{2}  \\
  \lesssim (1-\eta)\left(\|\overline{\partial}((X_{\eta}-\overline{X_{\eta}})^{k}N_{\delta,q}u)\|^{2} + \|\overline{\partial}^{*}((X_{\eta}-\overline{X_{\eta}})^{k}N_{\delta,q}u)\|^{2}\right) \\
  + terms\; acceptable \;terms  \\
  \lesssim (1-\eta)\left( \|[\overline{\partial},(X_{\eta}-\overline{X_{\eta}})^{k}]N_{\delta,q}u\|^{2}+\|[\overline{\partial}^{*},(X_{\eta}-\overline{X_{\eta}})^{k}]N_{\delta,q}u\|^{2} \right) \\
  + (1-\eta)\left(\|(X_{\eta}-\overline{X_{\eta}})^{k}\overline{\partial}N_{\delta,q}u\|^{2} + \|(X_{\eta}-\overline{X_{\eta}})^{k}\overline{\partial}^{*}N_{\delta,q}u\|^{2}\right)  \\
  + acceptable \;terms\;.
\end{multline}
For $(1-\eta)$ small enough, the fourth line in \eqref{est45} can be absorbed into the first line, giving
\begin{multline}\label{est46}
 \|[\overline{\partial},(X_{\eta}-\overline{X_{\eta}})^{k}]N_{\delta,q}u\|^{2}+\|[\overline{\partial}^{*},(X_{\eta}-\overline{X_{\eta}})^{k}]N_{\delta,q}u\|^{2}  
  + \delta\|[\nabla,(X_{\eta}-\overline{X_{\eta}})^{k}]N_{\delta,q}u\|^{2}  \\
  \lesssim (1-\eta)\left(\|(X_{\eta}-\overline{X_{\eta}})^{k}\overline{\partial}N_{\delta,q}u\|^{2} + \|(X_{\eta}-\overline{X_{\eta}})^{k}\overline{\partial}^{*}N_{\delta,q}u\|^{2}\right)  
  + acceptable \;terms\;.
\end{multline}

The left hand side of \eqref{est46} is what is needed to estimate the sum of the left hand sides of \eqref{est36}, \eqref{est37}, and \eqref{est38} (see the discussion form \eqref{eq31} to \eqref{39a}). Thus
\begin{multline}\label{est47}
 \|\overline{\partial}N_{\delta,q}u\|_{k,2h_{\eta}}^{2}+\|\overline{\partial}^{*}N_{\delta,q}u\|_{k,2h_{\eta}}^{2}+\delta\|\nabla N_{\delta,q}u\|_{k,2h_{\eta}}^{2} \\
  \lesssim (1-\eta)\left(\|(X_{\eta}-\overline{X_{\eta}})^{k}\overline{\partial}N_{\delta,q}u\|^{2} + \|(X_{\eta}-\overline{X_{\eta}})^{k}\overline{\partial}^{*}N_{\delta,q}u\|^{2}\right)  
  + acceptable \;terms\;,
\end{multline}
where `acceptable terms' stands for terms that are of order at most $(k-1)$ in $\overline{\partial}N_{\delta,q}u$, $\overline{\partial}^{*}N_{\delta,q}u$, or $\delta^{1/2}\nabla N_{\delta,q}u$, or are the same as the terms on the left hand side of \eqref{est47}, but with  a small constant in front (for $(1-\eta)$ small enough and then $\delta\leq \delta_{0}(\eta)$), or are dominated by $\|u\|_{k}^{2}$. After combining the estimates $\|(X_{\eta}-\overline{X_{\eta}})^{k}\overline{\partial}N_{\delta,q}u\|^{2}\lesssim \|\overline{\partial}N_{\delta,q}u\|_{k,2h_{\eta}}^{2}+C_{\eta}\|N_{\delta,q}u\|_{k-1}^{2}$ and $\|(X_{\eta}-\overline{X_{\eta}})^{k}\overline{\partial}^{*}N_{\delta,q}u\|^{2}\lesssim \|\overline{\partial}^{*}N_{\delta,q}u\|_{k,2h_{\eta}}^{2}+C_{\eta}\|N_{\delta,q}u\|_{k-1}^{2}$ with the induction hypothesis, and absorbing terms, we obtain that there are a constant $C$, $\eta\in(0,1)$, and $\delta_{0}$ such that 
\begin{equation}\label{est48}
  \|\overline{\partial}N_{\delta,q}u\|_{k,2h_{\eta}}^{2}+\|\overline{\partial}^{*}N_{\delta,q}u\|_{k,2h_{\eta}}^{2}+\delta\|\nabla N_{\delta,q}u\|_{k,2h_{\eta}}^{2} \leq C\|u\|_{k}^{2}\;,\;0<\delta\leq \delta_{0}\;.
\end{equation}
This completes the induction on $k$; \eqref{est48} holds for all $k\in\mathbb{N}$, with $C$, $\eta$, and $\delta_{0}$ depending on $k$.

\smallskip

Fix $k\in\mathbb{N}$. In view of \eqref{est34}, and choosing $\eta$ so that \eqref{est48} holds, we see that there is $\delta_{0}(k)>0$ and a constant $C_{k}$ such that
\begin{equation}\label{est49}
 \|N_{\delta,q}u\|_{k}^{2} \leq C_{k}\|u\|_{k}^{2}\;,\,0<\delta\leq\delta_{0}(k)\;.
\end{equation}
(The weighted and unweighted norms are equivalent, with constants depending on $\eta$; once $\eta$ is fixed, this dependence is no longer relevant.) As we said at the beginning of this section, letting $\delta\rightarrow 0$ transfers this estimate to $N_{q}$. Since $k$ was arbitrary, we have now shown that if $N_{q+1}$ satisfies Sobolev estimates for $q\geq q_{0}$, then so does $N_{q}$. This concludes the downward induction on the degree $q$, and completes the proof of Theorem \ref{main}.
\end{proof}

\bigskip
\bigskip
\bigskip
\providecommand{\bysame}{\leavevmode\hbox to3em{\hrulefill}\thinspace}

\end{document}